\documentclass[letterpaper, 10 pt, conference]{ieeeconf}  

\IEEEoverridecommandlockouts                              

\usepackage{amsmath} 
\usepackage{amssymb}  
\usepackage{mathtools}

\usepackage[normalem]{ulem}

\usepackage{graphicx}      
\usepackage{subcaption}    
\usepackage{float}         
\usepackage{acronym}       
\usepackage{cite}
\usepackage{pgfplots}
\pgfplotsset{compat=1.18} 
\usepackage{mathtools}
\usepackage{bbm}
\usetikzlibrary{calc}

\newcommand{\G}{\mathcal{G}}
\newcommand{\V}{\mathcal{V}}
\newcommand{\E}{\mathcal{E}}
\newcommand{\T}{\mathcal{T}}

\newcommand{\R}{\mathbb{R}}
\DeclareMathOperator{\Img}{Im}
\DeclareMathOperator{\Ker}{Ker}
\DeclareMathOperator{\rank}{rk}

\DeclareMathOperator{\col}{col}

\newtheorem{theorem}{Theorem}
\newtheorem{lemma}{Lemma}

\newtheorem{definition}{Definition}
\newtheorem{remark}{Remark}
\newtheorem{assumption}{Assumption}

\title{\LARGE \bf
Combinatorial Admissibility in  Control-Affine Networks
}

\author{Daniel Zelazo, \IEEEmembership{Senior Member, IEEE} and Louis Theran
\thanks{This work was supported by the Israel Science Foundation grant no. 453/24 and UK Research and Innovation though the grant
UKRI1112.}
\thanks{D. Zelazo is with the Stephen B. Klein Faculty of Aerospace Engineering, Technion-Israel Institute of Technology, Haifa 320003, Israel 
        {\tt\small dzelazo@technion.ac.il}}%
\thanks{L. Theran is with the School of Mathematics and Statistics, University of St Andrews,
Scotland {\tt\small lst6@st-andrews.ac.uk}}.
}

\begin{document}

\maketitle
\thispagestyle{empty}
\pagestyle{empty}

\begin{abstract}
We study synchronization of heterogeneous control-affine nonlinear agents interconnected through diffusive (relative-output) measurements.
We separate the design into an \emph{edge-space} step—specifying a
stabilizing model evolution for relative outputs—and a \emph{lift} step—realizing the prescribed edge motion using the agents' allowable input directions, constrained by the control-affine geometry of the agents.
We introduce an admissibility notion that characterizes
when an edge-driven diffusive design is feasible. 
\textcolor{black}{We derive checkable combinatorial certificates that connect graph
topology and actuation limits directly to admissibility, so that
feasible edge dynamics can be verified in a practical and transparent
way.} The results are illustrated on synchronization of nonlinear oscillators.
\end{abstract}

\section{Introduction}

One of the canonical approaches for the coordination of multi-agent systems is through diffusive coupling.  In this architecture, agents interact through relative information with their
neighbors, and a collective behavior, like synchronization or formation keeping, emerges from local rules. For integrator and linear agent models this viewpoint leads to a mature theory that
connects convergence to graph connectivity and Laplacian structure \cite{MesbahiEgerstedt2010}. For nonlinear agents, however, there is an additional layer that is not apparent in the linear setting. The coupling law specifies how
relative quantities should evolve, while the plant dynamics may constrain which relative motions can actually be produced through the available input
directions. Understanding this gap between “desired” and “realizable” relative behavior is central to synchronization beyond the linear regime.

Nonlinear synchronization has a rich history and has been
studied using passivity, contraction, and geometric methods;
see, e.g., \cite{Arcak2007, BurgerDePersis2015,
WielandSepulchreAllgower2011, ScardoviSepulchre2009,
PhamSlotine2007}. In geometric settings, where states evolve on nonlinear spaces (e.g., Lie groups or Riemannian manifolds), intrinsic distances and gradient flows provide natural coordination laws and reveal global topological constraints; see \cite{SarletteSepulchre2009, TronAfsariVidal2013}.

The geometric approach adopted here is inspired by the edge-space viewpoint introduced in our previous work \cite{Theran_IFAC26} to study formation control problems.  In \cite{Theran_IFAC26}, a geometric template that separates coordination into an \emph{edge-space design} step and a \emph{node-space realization} step was introduced. The starting point is a non-linear measurement map from configurations to edge measurements, whose image forms a feasible set in measurement space; an artificial edge dynamical system constrained to evolve on this feasible set (e.g., via a Riemannian gradient flow) is proposed, and then the \textcolor{black}{agent} dynamics are obtained by lifting the edge dynamics to node space.  The strength of this 
approach lies in the  fact that the edge dynamics are easier to analyze and yield a more flexible design space.  In 
\cite{Theran_IFAC26}, the \textcolor{black}{agent} dynamics are integrators, so any edge dynamics admit \emph{some} lifting: 
once a feasible edge velocity is specified, it can always be implemented by a suitable choice of \textcolor{black}{agent} velocities. 
The present note studies a complementary situation of a synchronization problem in which the 
 measurement map factors through a linear graph operator, but the \textcolor{black}{agent}
dynamics are constrained.  The main question is then how to design edge dynamics that are realizable.
The answer turns out to depend not just on the edge dynamics but on a structural compatibility
between the graph topology and the agents' actuation directions, a condition that is automatically 
satisfied for integrator \textcolor{black}{agent} dynamics but becomes the central obstruction in the control-affine 
case we study here.  

Our contributions are as follows. First, we formulate diffusive synchronization for heterogeneous non-linear control-affine agents in an explicit edge-space geometry: relative outputs evolve on a feasible set induced by the measurement map, and we design stabilizing edge dynamics directly in that space. Second, we show how to realize the desired edge evolution by a minimum-norm lifting law for control-affine \textcolor{black}{agent} dynamics, identifying a feasibility condition that we term \emph{admissibility}, that characterizes when the edge-driven design can be implemented by the available input directions. Third, we provide tractable, structure-based admissibility tests by relating generic feasibility to the structured rank of the lifted edge map and to maximum matchings in an associated bipartite graph, thereby making the interaction between graph topology and actuation constraints explicit. Finally, we illustrate the theory on a nonlinear oscillator synchronization example, where the matching certificate accurately predicts success and failure under two actuation patterns.


\paragraph*{Notation}
Let $\R$ denote the real numbers. For an integer $n$, $\mathbf 1_n\in\R^n$ is the all-ones vector, $I_n$ the $n\times n$ identity, and \textcolor{black}{$\bf 0$  the zero-vector}. For matrices, $\Img(\cdot)$, $\Ker(\cdot)$, and $\rank(\cdot)$ denotes the image, kernel, and rank, while $(\cdot)^\top$ transpose, and $(\cdot)^\dagger$ the Moore--Penrose pseudoinverse. For vectors $x_i$, $\col(x_1,\ldots,x_n)$ denotes vertical stacking, and $\mathrm{blkdiag}(\cdot)$ the block-diagonal operator. The Kronecker product is $\otimes$.


\section{Problem Setup}
\label{sec:problem_setup}

We consider a network of $N$ agents indexed by $\V:=\{1,\dots,N\}$. Agent $i$ has state
$x_i\in\mathbb{R}^{n_i}$, input $u_i\in\mathbb{R}^{p_i}$, and measured/regulated output
$y_i\in\mathbb{R}^{d}$. The dynamics are assumed to be \emph{control-affine},
\begin{equation}
\dot x_i = f_i(x_i) + G_i(x_i)\,u_i, 
\qquad y_i = h_i(x_i),
\label{eq:control_affine_agent}
\end{equation}
where $f_i:\mathbb{R}^{n_i}\to \mathbb{R}^{n_i}$, $G_i:\mathbb{R}^{n_i}\to \mathbb{R}^{n_i\times p_i}$,
and $h_i:\mathbb{R}^{n_i}\to \mathbb{R}^{d}$ are analytic. We stack the states, inputs, and outputs as
$x := \col(x_1,\dots,x_N)$, $
u := \col(u_1,\ldots,u_N)$, and $
y := \col(y_1,\ldots,y_N)=h(x)\in\mathbb{R}^{Nd}$.
We similarly stack the drift and input matrices as
$f(x):=\col(f_1(x_1),\ldots,f_N(x_N))$ and $
G(x):=\mathrm{blkdiag}(G_1(x_1),\ldots,G_N(x_N))$.

Agents exchange relative output information over a graph $\G=(\V,\E)$ with $|\E|=m$. For clarity, we begin with an undirected graph and assign an arbitrary orientation to each edge. Let $B\in\mathbb{R}^{N\times m}$ denote the incidence matrix. For an oriented edge $e=(i,j)$, the corresponding column $b_e$ has $-1$ in row $i$, $+1$ in row $j$, and $0$ elsewhere.

Define the \emph{edge disagreement} (relative output) signal as
\begin{equation*}
z \;:=\; (B^\top\!\otimes I_d)\,y \;\in\; \mathbb{R}^{md}.
\label{eq:edge_disagreement_z}
\end{equation*}
Then the control goal is \emph{output \textcolor{black}{synchronization}},
$\lim_{t\to\infty} \|z(t)\| = \textcolor{black}{\bf 0},$
i.e., $y_i(t)-y_j(t)\to \textcolor{black}{\bf 0}$ for every edge $(i,j)\in \E$. If $\G$ is connected, \textcolor{black}{synchronization} implies each agent $y_i(t)\to \bar y(t)$ for a common trajectory $\bar y(t)\in\mathbb{R}^d$.

Equivalently, \textcolor{black}{synchronization} corresponds to convergence to the \emph{\textcolor{black}{synchronization} manifold},
\begin{equation}
\mathcal{A} \;:=\; \{x:\ \textcolor{black}{F(x)}=(B^\top\!\otimes I_d)\,h(x)=\textcolor{black}{\bf 0}\} \;=\; F^{-1}(\textcolor{black}{\bf 0}),
\label{eq:agreement_manifold}
\end{equation}
where $F:\mathbb{R}^{\sum n_i}\to \mathbb{R}^{md}$ is the \emph{edge map} defined as
\begin{equation}
F(x) := (B^\top\!\otimes I_d)\,h(x).
\label{eq:edge_map_F}
\end{equation}

We focus on distributed controllers that depend on the local state and relative outputs,
\textcolor{black}{
$
u_i=\kappa_i\big(x_i, \mathrm{col}_{j \in \mathcal N_i}(y_j - y_i) \big),$ where $\kappa_i :\mathbb{R}^{n_i} \times \mathbb{R}^{d|\mathcal N_i|} \to \mathbb{R}^{p_i}$,
}
and $\mathcal{N}_i$ denotes the neighbor set of \textcolor{black}{vertex} $i$ in $\G$. A canonical static diffusive coupling architecture takes the form
\begin{equation*}
u_i = \phi_i(x_i)\;-\;\sum_{j\in\mathcal{N}_i} a_{ij}\,\psi\!\big(y_i-y_j\big),
\label{eq:diffusive_general}
\end{equation*}
with weights $a_{ij}>0$ and typically $\psi(\cdot)$ odd and monotone (e.g., $\psi(\eta)=\eta$). When 
$\phi_i$ is zero and $\psi$ is the identity, the coupling is a graph Laplacian.


\section{Edge-Space Geometry and a Model System}
\label{sec:edge_geometry_model}
We begin with the edge map introduced in \eqref{eq:edge_map_F}, which relates the state of each agent to the edge state as $z=F(x)$. Let
\begin{equation*}
Q := \mathrm{Im}\,F \subseteq \mathbb{R}^{md}
\label{eq:feasible_edge_set}
\end{equation*}
denote the feasible set of edge disagreements. Define the Jacobian of $F$ by
\begin{equation}
J(x) := DF(x) = (B^\top\!\otimes I_d)\,Dh(x),
\label{eq:J_def}
\end{equation}
where \textcolor{black}{$D$ denotes the Jacobian operator} and $Dh(x)=\mathrm{blkdiag}(Dh_1(x_1),\dots,Dh_N(x_N))$ .

\begin{assumption}[Constant-rank neighborhood]
\label{ass:constant_rank}
There exists an open neighborhood 
\textcolor{black}{$\Omega\subseteq \mathbb{R}^{\sum_i n_i}$}
of the \textcolor{black}{synchronization} set $\mathcal{A}=F^{-1}(0)$ such that
$\rank(J(x))$ is constant for all $x\in\Omega$.
\end{assumption}

Under Assumption~\ref{ass:constant_rank}, the image $Q_\Omega:=F(\Omega)$ is an immersed (and locally embedded)
submanifold of $\mathbb{R}^{md}$, and for $z=F(x)$ with $x\in\Omega$ the tangent space satisfies
%
\begin{equation*}
T_zQ_\Omega = \mathrm{Im}\,J(x).
\label{eq:tangent_space_Q}
\end{equation*}
In the sequel we restrict attention to $\Omega$ and write $Q$ in place of $Q_\Omega$.

Following the same ideas as from \cite{Theran_IFAC26}, we consider an artificial system of edges with state $z\in \mathbb{R}^{md}$ with integrator dynamics,
\begin{equation}
\dot z = v.
\label{eq:edge_integrator}
\end{equation}
We aim to design a stabilizing edge feedback $v^\star(z)$ on $Q$ that drives $z\to 0$.  In this direction, we consider the Riemannian gradient flow for $V_e(z)=\tfrac12\|z\|^2$, leading to
\begin{equation}
v^\star(F(x)) = -\Pi(x)\,F(x), 
\label{eq:edge_gradient_flow}
\end{equation}
where $\Pi(x):\mathbb{R}^{md}\to T_{F(x)}Q$ is the orthogonal projector onto $T_{F(x)}Q$. To avoid ambiguity when multiple $x$ map to the same $z$, we write the projector as a function of $x$,
\begin{equation*}
\Pi(x) := J(x)J(x)^\dagger,
\label{eq:Pi_def}
\end{equation*}
so that $\Pi(x)$ is the orthogonal projector onto $\Img(J(x))=T_{F(x)}Q$.

We now must map the edge flow dynamics in \eqref{eq:edge_gradient_flow} to the agent dynamics of the system.
Differentiating $z=F(x)$ along \eqref{eq:control_affine_agent} yields
\begin{equation*}
\dot z = J(x)f(x) + J(x)G(x)\,u.
\label{eq:zdot_lift_eq}
\end{equation*}
To realize the desired edge velocity $v^\star(z)$ in \eqref{eq:edge_gradient_flow}, we require
\begin{equation}
J(x)G(x)\,u = v^\star(F(x)) - J(x)f(x).
\label{eq:lift_constraint}
\end{equation}
Among all feasible inputs, we select the minimum-norm lift,
\begin{equation}
\begin{aligned}
u^\star(x)
&:= \arg\min_{u}\ \|u\|^2 \quad \text{s.t. } \eqref{eq:lift_constraint}. 
\end{aligned}
\label{eq:min_norm_lift}
\end{equation}
When the constraint is feasible, this yields the explicit \emph{model controller},
\begin{align}
u^\star(x) &= \big(J(x)G(x)\big)^\dagger\Big(v^\star(F(x)) - J(x)f(x)\Big)\nonumber\\
&=\big(J(x)G(x)\big)^\dagger\Big(-\Pi(x)\,F(x) - J(x)f(x)\Big).
\label{eq:model_controller_u_star}
\end{align}

The remainder of the paper analyzes (i) conditions under which \eqref{eq:min_norm_lift} is feasible and yields closed-loop \textcolor{black}{synchronization}, and (ii) how to obtain distributed approximations and directed variants, culminating in a combinatorial admissibility criterion.

\section{Edge-Driven Controller Families and Distributed Implementations}
\label{sec:family_distributed_undirected}
This section develops an \emph{edge-driven controller family} that extends the model controller \eqref{eq:model_controller_u_star}.
The key observation is that the lift constraint \eqref{eq:lift_constraint} is typically \emph{underdetermined}: many inputs can realize the same desired edge velocity.
We exploit this freedom to (i) parameterize a family of controllers consistent with a prescribed edge flow, and
(ii) identify distributed specializations that recover classical diffusive coupling architectures.

\subsection{A family of lifts for a prescribed edge flow}
\label{subsec:lift_family}
Let 
\begin{subequations}\label{eq:Ax_bx_def}
\begin{align}
A(x)&:=J(x)G(x)\in\mathbb{R}^{md\times p}\label{eq:Ax}\\
b(x)&:=v^\star(F(x)) - J(x)f(x)\in\mathbb{R}^{md}\label{eq:bx}
\end{align}
\end{subequations}
where $p:=\sum_i p_i$.
Then the lift constraint \eqref{eq:lift_constraint} is the linear equation
\begin{equation}
A(x)u=b(x).
\label{eq:Au_equals_b}
\end{equation}
Whenever \eqref{eq:Au_equals_b} is feasible, the set of all solutions is an affine space.
In particular, the Moore--Penrose solution \eqref{eq:model_controller_u_star} is the unique solution with minimum Euclidean norm, and the general solution is
\begin{equation*}
u(x) = u^\star(x) + \big(I - A(x)^\dagger A(x)\big)\,w(x),
\label{eq:general_solution_family}
\end{equation*}
where $w(x)\in\mathbb{R}^{p}$ is an arbitrary (possibly state-dependent) signal.
The matrix $\mathcal{M}(x):=I-A(x)^\dagger A(x)$ is the orthogonal projector onto $\Ker(A(x))$.
Thus, the term $\mathcal{M}(x)w(x)$ does not affect the induced edge velocity $\dot z$ and can be used
to enforce secondary objectives (e.g., input saturation handling, additional regulation tasks, or distributed implementability).

\begin{remark}[Edge-driven family in ``gain'' form]
\label{rem:gain_form_family}
In the special case $f\equiv 0$ (or after exact drift compensation), one can write a simpler edge-driven family,
\begin{equation*}
u(x) = -\nu(x)\,F(x),
\label{eq:gain_form_u}
\end{equation*}
where $\nu(x)\in\mathbb{R}^{p\times md}$ is chosen so that the induced edge dynamics satisfy
$\dot z = -\eta(x)z$ with $\eta(x)=A(x)\nu(x)$.
The model controller corresponds to $\nu(x)=A(x)^\dagger\Pi(x)$.
Distributed realizations (below) correspond to sparse choices of $\nu(x)$. \hfill $\Diamond$
\end{remark}

The edge-space design $v^\star$ is realizable by the physical agents when the required right-hand side lies in the range of $A(x)$.
In particular, we require that inputs can generate \emph{all feasible} edge directions, captured by the map $J(x)$.
This will be formalized as \emph{admissibility} and linked to combinatorial certificates in subsequent sections.


\subsection{Distributed implementation for relative degree-one dynamics}
\label{subsec:distributed_outer_loop}
We now identify a broad and practically important class of distributed controllers obtained by combining
\emph{local right inversion} of the agent output dynamics with a \emph{diffusive outer loop}.
In this direction, we further assume each agent admits relative-degree-one output 
dynamics~\textcolor{black}{\cite{Khalil2002}}, expressed as
\begin{equation}
\dot y_i = a_i(x_i) + H_i(x_i)u_i,
\label{eq:ydot_again}
\end{equation}
where $H_i(x_i)\in\mathbb{R}^{d\times p_i}$ is right invertible on the region of interest.
Define a virtual input $v_i\in\mathbb{R}^d$ and apply the minimum-norm local right inverse
\begin{equation}
u_i = H_i(x_i)^\dagger\big(v_i-a_i(x_i)\big),
\label{eq:local_inversion_again}
\end{equation}
which enforces $\dot y_i=v_i$.  In other words, we can feedback linearize the output dynamics to generate integrator dynamics.  This now leads to our first result.


\begin{theorem}\label{thm:distributed_outer_loop_consensus}
Suppose each agent admits relative-degree-one output dynamics \eqref{eq:ydot_again} on a forward-invariant set
$\mathcal X_i$,
and $H_i(x_i)$ is right invertible for all $x_i\in\mathcal X_i$.
Apply the local right-inversion controller \eqref{eq:local_inversion_again} and let the virtual input be chosen by the diffusive outer loop
\begin{equation}\label{eq:outer_loop_diffusive_again}
v_i \;=\; -\sum_{j=1}^N w_{ij}\big(y_i-y_j\big),
\end{equation}
 with weights $w_{ij}\ge 0$ 
of a (directed) graph $\G$.
If $\G$ contains a (directed) spanning tree, then the outputs reach consensus as $t\to\infty$.
\end{theorem}


\textcolor{black}{The result illustrates that 
the proposed framework recovers the classical diffusive coupling
structure for relative-degree-one systems; see, e.g.,
\cite{Chopra2008} for a proof.}

\begin{remark}[Relation to the edge-driven family]
\label{rem:relation_family_diffusive}
The distributed controller \eqref{eq:local_inversion_again}--\eqref{eq:outer_loop_diffusive_again} can be interpreted as a sparse realization of an edge-driven law: the outer loop prescribes a feasible edge evolution (through $v$), while the local inversions realize it agentwise without forming global pseudoinverses such as $(J(x)G(x))^\dagger$.
This provides a bridge between classical diffusive coupling and the edge-space design-and-lift template. \hfill $\Diamond$
\end{remark}

\subsection{Admissibility of edge-driven designs}
\label{subsec:admissibility}

The model-based construction developed above designs a desired edge velocity
$v^\star(F(x))$ on the feasible edge set and then selects an input $u$ by enforcing the instantaneous constraint
\eqref{eq:lift_constraint}. This subsection formalizes when that constraint is feasible and clarifies the closed-loop
implications.

\begin{definition}[Admissibility]
\label{def:admissibility}
    Let $\Omega$ be the constant-rank neighborhood from Assumption~\ref{ass:constant_rank}. Define
$J(x):=DF(x)$ and $A(x):=J(x)G(x)$ as in \eqref{eq:J_def} and \textcolor{black}{\eqref{eq:Ax}}.  The agent--network pair is 
\begin{enumerate}
\item \emph{exactly admissible on $\Omega$} if
\begin{equation}
\Img\big(A(x)\big)=\Img\big(J(x)\big), \text{ for all $x\in\Omega$;}
\label{eq:admissibility_exact}
\end{equation}

\item  \emph{locally generically 
admissible} if, \eqref{eq:admissibility_exact} holds for almost all 
$x\in \Omega$;
\item \emph{generically admissible} if
\eqref{eq:admissibility_exact} holds for almost all 
$x$.
\end{enumerate}
\end{definition}

Condition \eqref{eq:admissibility_exact} states that the inputs can generate \emph{every feasible infinitesimal edge motion}:
since feasible edge velocities satisfy $\dot z\in T_{F(x)}Q=\Img(J(x))$, admissibility ensures that any such $\dot z$ can be
realized through the control channel $G(x)$.

\begin{remark}[Relationship between admissibility concepts]
Because $A(x)$ is the restriction of $J(x)$ to the image of $G(x)$, 
\eqref{eq:admissibility_exact} is equivalent to the transversality 
statement $\Img G(x) + \Ker J(x) = \mathbb{R}^{\sum n_i}$, which, 
in turn is equivalent to the ranks of $A(x)$ and $J(x)$ being equal.  The 
failure of transversality is expressible as a polynomial in $x$, 
and so, if $G(x)$ and $J(x)$ are defined by analytic functions, 
if there is one $x$ at which transversality holds then there is 
an $\Omega\ni x$ such that exact admissibility holds on $\Omega$ and, 
moreover, the controller is generically admissible.  In particular, in 
these cases, local generic admissibility and admissibility are 
equivalent concepts.  For more general classes of functions, such as 
smooth functions, the two may be different.
\hfill $\Diamond$
\end{remark}

\begin{remark}[Exact vs.\ least-squares lifting]
\label{rem:exact_vs_ls}
When $b(x)\in\Img(A(x))$, the constraint $A(x)u=b(x)$ is feasible and the minimum-norm lift \eqref{eq:model_controller_u_star}
enforces the desired edge velocity exactly. When $b(x)\notin\Img(A(x))$, the same pseudoinverse expression yields the
minimum-norm least-squares solution, and the realized edge velocity is the orthogonal projection of $b(x)$ onto $\Img(A(x))$. \hfill $\Diamond$
\end{remark}

\begin{lemma}[Exact realization under admissibility]
\label{lem:realize_edge_flow}
Suppose Assumption~\ref{ass:constant_rank} holds on $\Omega$ and let $\Pi(x):=J(x)J(x)^\dagger$.
Consider the projected-gradient model choice $v^\star(F(x))=-\Pi(x)F(x)$ and the lifted controller
\eqref{eq:model_controller_u_star} written as
\begin{equation}
u^\star(x) = A(x)^\dagger\big(v^\star(F(x)) - J(x)f(x)\big).
\label{eq:u_star_A_form}
\end{equation}
If the admissibility condition \eqref{eq:admissibility_exact} holds at $x\in\Omega$, then the lift constraint
\eqref{eq:lift_constraint} is feasible at $x$ and the closed-loop edge dynamics satisfy
\begin{equation}
\dot z = v^\star(F(x)) = -\Pi(x)\,z,\qquad z=F(x),
\label{eq:closed_loop_edge_flow}
\end{equation}
at that point.  
\textcolor{black}{Moreover, there exists a neighborhood 
$U \subseteq { Q}$ of $\mathbf 0$ such that, for trajectories remaining in 
$U$, the origin $z = \mathbf 0$ is locally exponentially stable.}
\end{lemma}

\begin{proof}
By construction, $v^\star(F(x))=-\Pi(x)F(x)\in\Img(J(x))$ because $\Pi(x)$ projects onto $\Img(J(x))$.
Under admissibility, $\Img(J(x))= \Img(A(x))$, hence $b(x)=v^\star(F(x))-J(x)f(x)\in \Img(A(x))$ and the constraint is feasible.
For a feasible linear equation $A(x)u=b(x)$, the pseudoinverse yields a solution satisfying $A(x)u^\star=b(x)$, which combined with
$\dot z = J(x)f(x)+A(x)u$ gives \eqref{eq:closed_loop_edge_flow}.

Along trajectories that 
remain in $\Omega$ and satisfy admissibility pointwise,
the edge Lyapunov function $V_e(z)=\tfrac12\|z\|^2$ satisfies 
\textcolor{black}{$$\dot V_e(z) = -\|\Pi(x)z\|^2 \leq 0.$$
On a sufficiently small neighborhood $U_1$ of ${\bf 0}$ in ${Q}$, 
$\|\Pi({\bf 0})z\| \ge \frac{1}{2}\|z\|$, because tangent vectors 
in $T_{\bf 0}{Q}$ are limits of the direction vectors of secants
(i.e., there is an $\varepsilon > 0$ so that if $0 < \|z\| < \varepsilon$, 
the  distance of $z/\|z\|$ to $T_{\bf 0} Q$ is less than $\frac{1}{2}$).
Independently, there is a 
neighborhood $U_2$ of ${\bf 0}$ in ${Q}$ such that 
$\|\Pi({\bf 0})z - \Pi(x)z\| \le \frac{1}{4}\|z\|$.  
Defining $U=U_1\cap U_2$ , we have for all 
$z \in U$, $\|\Pi(x)z\| \ge \frac{1}{4}\|z\|$.
On this neighborhood, we get the estimate 
\[
\dot V_e(z) = -\|\Pi(x)z\|^2 \le -\frac{1}{16}\|z\|^2<0.
\]
An application of Grönwall's inequality \cite{Khalil2002} then implies that 
$\|\Pi(x)z\|\to 0$  exponentially quickly.  With 
the lower bound, this further  implies that 
$z\to 0$ exponentially quickly.
}
\end{proof}

 Lemma~\ref{lem:realize_edge_flow} shows that admissibility is the precise 
 condition ensuring the physical closed-loop reproduces the
 designed (projected) edge gradient flow. 
\textcolor{black}{The above argument establishes local exponential convergence of 
$z$ to ${\bf 0}$ on a neighborhood $U \subseteq  Q$.
Now we establish that there is a forward invariant neighborhood of $\mathcal A$. 
Define, for $\rho > 0$, the closed 
sublevel set $\overline{U_\rho} := \{ z \in  Q : V_e(z) \le \rho \} \subseteq U$.
Then define $\Omega_\rho := \Omega \cap F^{-1}(\overline{U_\rho})$. Since 
${\cal A}$ is closed and $\Omega_{\rho=0} = {\cal A}$, for 
a sufficiently small $\rho > 0$, $\overline{\Omega_\rho}\subseteq \Omega$.
Since $\dot V_e(z) \le 0$ when $z\in U$,
trajectories initialized in
$\Omega_\rho$ satisfy $F(x(t)) \in U_\rho$ for all $t \ge 0$.
Since trajectories in $\Omega_\rho$ cannot reach the boundary $\partial \Omega$ without 
leaving $\overline{\Omega_\rho}$, we conclude that $\Omega_\rho$ is forward 
invariant. 
}

Since $\mathcal A=F^{-1}(\textcolor{black}{\bf 0})$ is a smooth embedded submanifold, 
fix a compact neighborhood $K\subset\Omega$ of a point $x^\star\in\mathcal A$.
By continuity, $J(x)$ is bounded on $K$. Therefore, under admissibility and exponential 
stability of the edge model, the standard ISS
argument of \textcolor{black}{\cite[Section 3.3]{Theran_IFAC26}} yields local exponential 
convergence of the lifted agent dynamics to $\mathcal A$.


\begin{remark}[Why Theorem~\ref{thm:distributed_outer_loop_consensus} is always admissible]
\label{rem:rd1_implies_admissible}
In the relative-degree-one specialization \eqref{eq:ydot_again}--\eqref{eq:local_inversion_again}, 
the controller satisfies the stronger agent-by-agent transversality condition 
$\Img G_i(x_i) + \Ker {\rm d}_{x_i}h_i = \mathbb{R}^{n_i}$ because $\Img H_i(x_i) = \mathbb{R}^d$.
Since the graph operator is linear, this implies admissibility at $x_i$.
\hfill $\Diamond$
\end{remark}

For underactuated agents or restricted actuation directions, admissibility can fail even on connected graphs.
In the next subsection we derive checkable sufficient conditions for (generic) admissibility in terms of structured rank and maximum matchings
of an associated bipartite graph. 

\subsection{Generic admissibility and combinatorial certificates}
\label{subsec:combinatorial_admissibility}

The admissibility condition in Definition~\ref{def:admissibility} is analytic and state-dependent through $J(x)$ and $G(x)$.
In this subsection we derive \emph{structure-based} 
sufficient conditions that can be checked combinatorially,
starting with the undirected setting.

Fix a region $\Omega$ and suppose that the zero/non-zero pattern of $Dh_i(x_i)G_i(x_i)$ is constant on $\Omega$.  We assume the controller is 
analytic, so that $A(x)$ achieves its maximal rank for almost all $x\in \Omega$.
We now identify a structural property of the graph $\G$ that 
implies the controller is generically admissible.




For a connected, undirected graph, $\G$ the edge disagreement 
vector $z=F(x)$ must lie in an $(N-1)d$-dimensional linear 
subspace cut out by cycle relations (i.e., the sum of its components around any cycle is zero).  Hence, 
the generic rank of $A(x)$ is at most 
$(N-1)d$.  We describe a combinatorial 
condition that implies equality, from 
which generic admissibility follows.

Let $\T\subseteq \E$ be a spanning tree with $|\T|=N-1$, and let
$B_\T\in\R^{N\times(N-1)}$ denote the incidence matrix restricted to the edges of $\T$ (with arbitrary orientation).
Define the tree-edge disagreement map,
\[
F_\T(x):=(B_\T^\top\!\otimes I_d)h(x),
\]
 and its Jacobian
\[J_\T(x):=DF_\T(x)=(B_\T^\top\!\otimes I_d)\,Dh(x).
\]
Since $B_\T$ has rank $N-1$, the tree-edge coordinates span the disagreement subspace; in particular,
$\rank(J_\T(x))=\rank(J(x))$ generically, and it suffices to certify that $A_\T(x):=J_\T(x)G(x)$ has rank $(N-1)d$ generically.

To this end, we introduce the following assumption.
\begin{assumption}[Generic independence of nonzero entries]
\label{ass:generic_independence}
For any spanning tree $\T$ of $\G$, the 
non-zero entries in $A_\T(x)$ can be varied 
independently by changing $x$.
\end{assumption}

\textcolor{black}{Assumption \ref{ass:generic_independence} is a standard structural-rank condition ensuring that the generic rank of $A_\T(x)$ is determined by its sparsity pattern, enabling the use of matching-based certificates. In this direction, }
%
%
%
construct a bipartite graph $\mathcal{H}_\T=(\mathcal{L},\mathcal{R},\mathcal{E})$ using the 
zero/non-zero pattern of 
$A_\T(x)$ as follows:
\begin{itemize}
\item Left vertices $\mathcal{L}$ index the $(N-1)d$ tree-edge coordinates:
$\mathcal{L}:=\{(e,k): e\in \T,\ k\in\{1,\dots,d\}\}$.

\item Right vertices $\mathcal{R}$ index the input channels: $\mathcal{R}:=\{(i,\ell): i\in \V,\ \ell\in\{1,\dots,p_i\}\}$.
\item An edge $((e,k),(i,\ell))\in\mathcal{E}$ is present iff the entry of $A_\T(x)$ corresponding to row $(e,k)$ and column $(i,\ell)$
is structurally nonzero.
\end{itemize}

\begin{definition}[Matching]
A \emph{matching} in a bipartite graph $\mathcal H=(\mathcal L,\mathcal R,\mathcal E)$ is a set of edges no two of which share a common
endpoint. A \emph{maximum matching} is a matching of maximum cardinality with size denoted by $\nu(\mathcal H)$.
\end{definition}



\begin{theorem}[Matching certificate for generic admissibility]
\label{thm:matching_certificate}
Let $\G$ be connected. Under Assumption~\ref{ass:generic_independence}, if
there is a spanning tree $\T$ of $\G$ such that 
\begin{equation}
\nu(\mathcal{H}_\T)=(N-1)d,
\label{eq:matching_full}
\end{equation}
then $\rank(A_\T(x))=(N-1)d$ for almost all $x\in\Omega$. Consequently,
\[
\rank(A(x))=\rank(J(x))=(N-1)d
\]
for almost all $x\in\Omega$, i.e., the system is generically admissible on $\Omega$.
\end{theorem}
\begin{proof}
The upper bound on the rank of $J(x)$ was 
discussed above.  By 
\cite[Prop. 2.4]{Whiteley1989MatroidHypergraphs}, 
which applies because of Assumption 
\ref{ass:generic_independence}, 
for any fixed $\T$, the rank of $A_\T(x)$
is equal to $\nu(\mathcal{H}_\T)$ for almost 
all $x$.  Hence, we get a matching lower 
bound for the rank of $A_\T(x)$, if $\nu(\mathcal{H}_\T)=(N-1)d$ for 
some $\T$.  In this case, we have 
\[
    (N-1)d\le \rank A_\T(x) \le \rank A(x) \le \rank J(x)
    \le (N-1)d,
\]
so equality holds throughout and we are done.
\end{proof}


\begin{remark}[Hall condition and design interpretation]
A matching of size $(N-1)d$ exists iff Hall's condition holds \cite{BondyMurty2008}:
$|\Gamma(\mathcal{S})|\ge |\mathcal{S}|$ for every $\mathcal{S}\subseteq \mathcal{L}$, where $\Gamma(\mathcal{S})$ is the neighbor set in
$\mathcal{H}_\T$. This reads as: every subset of disagreement coordinates must be influenced by at least as many independent input channels. \hfill $\Diamond$
\end{remark}

Theorem~\ref{thm:matching_certificate} gives a purely combinatorial sufficient condition for (generic) feasibility of the edge lift.
In Section~\ref{sec:simulation} we illustrate how the matching fails when a row of a local actuation map is structurally zero, leading to
persistent disagreement in the corresponding output component.

\section{Simulation: Admissible vs.\ Non-Admissible Edge Actuation on a Limit Cycle}
\label{sec:simulation}

We illustrate the admissibility results on a synchronization problem for nonlinear oscillators. 
We compare two actuation patterns on the same network: one that is admissible and achieves synchronization,
and one that is non-admissible.

In this direction, {\color{black}{we consider three identical nonlinear oscillators coupled over the undirected path graph $1$--$2$--$3$.
Each agent has state $x_i\in\R^2$ with $x_i := \begin{bmatrix} r_i& \theta_i \end{bmatrix}\textcolor{black}{^\top},
\; r_i>0,\ \theta_i\in\mathbb S^1$, and evolves according to the control-affine dynamics \eqref{eq:control_affine_agent}.
The drift $f$ is chosen so that $r_i=1$ is a stable limit cycle and $\theta_i$ rotates at constant speed $\omega>0$:
\begin{equation*}
f(x_i)=
\begin{bmatrix}
(1-r_i^2)r_i\\
\omega
\end{bmatrix}.
\label{eq:sim_drift_control_affine}
\end{equation*}

The control input is the \emph{same signal type} in both cases,
$u_i:=\begin{bmatrix}u_{r,i}& u_{\theta,i}\end{bmatrix}\textcolor{black}{^\top}\in\R^2$,
and the difference between the two models is encoded entirely in $G_i$.}}

Each oscillator is observed through its Cartesian position on the plane,
\begin{equation}
y_i = h(x_i) :=
\begin{bmatrix}
r_i\cos\theta_i&
r_i\sin\theta_i
\end{bmatrix}\textcolor{black}{^\top}\in\R^2,
\label{eq:sim_output_map}
\end{equation}
and we stack $y=\col(y_1,y_2,y_3)\in\R^{6}$. Synchronization corresponds to $y_i(t)-y_j(t)\to 0$ for all $i,j$.

\paragraph{Case A (admissible)}
We take the control input as $u_i:=\col(u_{r,i},u_{\theta,i})\in\R^2$ and select $G_i^{(A)} = I_2,\, i=1,2,3$.
Thus both radial and angular channels are available, and the input can directly modify both the radius and the phase of each oscillator.

\paragraph{Case B (non-admissible)}
We use the same input signal $u_i:=\col(u_{r,i},u_{\theta,i})\in\R^2$ but restrict actuation to the radial direction by choosing $G_i^{(B)}=\left[\begin{smallmatrix}1&0\\0&0\end{smallmatrix}\right],\, i=1,2,3 $.
Equivalently, the angular channel is unavailable and is discarded by $G_i^{(B)}$. In particular, $\dot\theta_i=\omega$ for all $i$,
so phase offsets cannot be corrected by feedback and planar synchronization fails for generic initial conditions.

Let $F(y):=(B^\top\!\otimes I_2)y$ denote the edge map introduced earlier (specialized here to the linear relative-measurement setting),
and define the edge disagreement as
$z := F(y) = (B^\top\!\otimes I_2)\,y \in\R^4$.
Following the edge-space construction of Section~\ref{sec:edge_geometry_model} (model integrator \eqref{eq:edge_integrator} and
projected-gradient choice \eqref{eq:edge_gradient_flow}), we prescribe the virtual edge dynamics $\dot z = v^\star(z)$.
In this example the feasible edge set is the linear subspace $Q=\Img(F)$, so the orthogonal projector onto $T_zQ$ is the identity.
Consequently, the projected Riemannian gradient flow for $V_e(z)=\tfrac12\|z\|^2$ reduces to the Euclidean gradient flow $v^\star(z) = -k\,z,\qquad k>0$,
which is exactly the edge-space ``Laplacian'' feedback. Admissibility asks  whether the control directions available to the oscillators span the edge-velocity directions required by
above, i.e., whether the lift constraint admits an exact solution so that the closed-loop satisfies $\dot z=v^\star(z)$.
When this fails, the pseudoinverse lift yields the closest achievable edge velocity (the orthogonal projection of $v^\star(z)$ onto the
achievable edge-velocity subspace), which is precisely what the matching test in Section~\ref{subsec:combinatorial_admissibility} certifies.

Because the graph is a tree and the output dimension is $d=2$, there are $(n-1)d=4$ independent disagreement coordinates. The admissibility question is whether the available control
directions can realize these four independent edge-space directions.

For the path graph $1$--$2$--$3$, the tree edge map is $F_\tau(x)=(B_\tau^\top\!\otimes I_2)h(x)$ with
$h_i(r_i,\theta_i)$ given in \eqref{eq:sim_output_map}. 
Differentiating yields $J_\tau(x)=DF_\tau(x)=(B_\tau^\top\!\otimes I_2)Dh(x)$, where
\begin{equation}\label{eq:sim_J}
Dh_i(r_i,\theta_i)=
\begin{bmatrix}
\cos\theta_i & -r_i\sin\theta_i\\
\sin\theta_i & \ \ r_i\cos\theta_i
\end{bmatrix}
=:J_i.
\end{equation}
The tree lift matrix is $A_\tau(x):=J_\tau(x)G$, where $G=\mathrm{blkdiag}(G_1,G_2,G_3)$ encodes the available inputs.

\emph{Case A (radial and tangential actuation).}
Now, we compute the lift matrix as
\begin{equation}
A_\tau^{(A)}(x)=
\begin{bmatrix}
 -J_1 & J_2 & 0\\
 0 & -J_2 & J_3
\end{bmatrix}\in\R^{4\times 6}.
\label{eq:Atau_caseA}
\end{equation}
Thus the first edge block (rows for $12$) depends only on \textcolor{black}{vertex}-$1$ and \textcolor{black}{vertex}-$2$ inputs, and the second edge block (rows for $23$) depends
only on \textcolor{black}{vertex}-$2$ and \textcolor{black}{vertex}-$3$ inputs; each $2\times 2$ block is generically dense.  The bipartite matching graph is  shown in Fig. \ref{fig:match_caseA_side} and corresponding trajectories in Fig. \ref{fig:sim_phase_A}.  

\emph{Case B (radial-only actuation).}
We use the same input signal $u_i=\col(u_{r,i},u_{\theta,i})\in\R^2$ as in Case~A, but restrict actuation by choosing $G_i^{(B)}$, so the angular channel is unavailable. 
%
The tree lift matrix $A_\tau(x)=J_\tau(x)G^{(B)}$ is computed as
\begin{equation}
A_\tau^{(B)}(x)=
\begin{bmatrix}
 -a_1 & 0 & a_2 & 0 & 0 & 0\\
 0 & 0 & -a_2 & 0 & a_3 & 0
\end{bmatrix}\in\R^{4\times 6},
\label{eq:Atau_caseB}
\end{equation}
where $a_i:=\begin{bmatrix}\cos\theta_i& \sin\theta_i\end{bmatrix}\textcolor{black}{^\top}\in\R^2$, and each displayed block is a $2\times 1$ column and the zero columns correspond to the unavailable inputs $u_{\theta,i}$.
Consequently the bipartite graph associated with \eqref{eq:Atau_caseB} contains the isolated vertices
$u_{\theta,1},u_{\theta,2},u_{\theta,3}$ (Fig.~\ref{fig:match_caseB_side}), and the maximum matching satisfies
$\nu\le 3<4=(N-1)d$. Thus $\rank(A_\tau^{(B)}(x))< (N-1)d$ generically and the lift constraint is generically infeasible.
The matching condition predicts admissibility in Case~A and failure in Case~B, which is confirmed in Fig.~\ref{fig:sim_phase_B}.

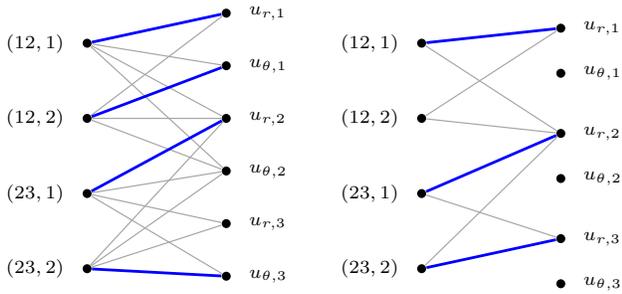
\begin{figure}[!h]
\centering
\begin{subfigure}[t]{0.485\columnwidth}
\centering
\begin{tikzpicture}[x=1cm,y=1cm,baseline]
  \tikzset{
    dot/.style={circle,fill,inner sep=1.2pt},
    Llab/.style={anchor=east,font=\scriptsize},
    Rlab/.style={anchor=west,font=\scriptsize},
    edge/.style={line width=0.45pt,draw=black!35},
    match/.style={line width=1.05pt,draw=blue}
  }

  \def\xL{0.0}
  \def\xR{1.85}
  \def\dxLab{0.18}

  \node[dot] (L1) at (\xL,3.0) {}; \node[Llab] at ($(L1)+(-\dxLab,0)$) {$(12,1)$};
  \node[dot] (L2) at (\xL,2.0) {}; \node[Llab] at ($(L2)+(-\dxLab,0)$) {$(12,2)$};
  \node[dot] (L3) at (\xL,1.0) {}; \node[Llab] at ($(L3)+(-\dxLab,0)$) {$(23,1)$};
  \node[dot] (L4) at (\xL,0.0) {}; \node[Llab] at ($(L4)+(-\dxLab,0)$) {$(23,2)$};

  \node[dot] (R1) at (\xR,3.4) {}; \node[Rlab] at ($(R1)+(\dxLab,0)$) {$u_{r,1}$};
  \node[dot] (R2) at (\xR,2.7) {}; \node[Rlab] at ($(R2)+(\dxLab,0)$) {$u_{\theta,1}$};
  \node[dot] (R3) at (\xR,2.0) {}; \node[Rlab] at ($(R3)+(\dxLab,0)$) {$u_{r,2}$};
  \node[dot] (R4) at (\xR,1.3) {}; \node[Rlab] at ($(R4)+(\dxLab,0)$) {$u_{\theta,2}$};
  \node[dot] (R5) at (\xR,0.6) {}; \node[Rlab] at ($(R5)+(\dxLab,0)$) {$u_{r,3}$};
  \node[dot] (R6) at (\xR,-0.1) {}; \node[Rlab] at ($(R6)+(\dxLab,0)$) {$u_{\theta,3}$};

  \foreach \LL in {L1,L2}{
    \foreach \RR in {R1,R2,R3,R4}{ \draw[edge] (\LL)--(\RR); }
  }
  \foreach \LL in {L3,L4}{
    \foreach \RR in {R3,R4,R5,R6}{ \draw[edge] (\LL)--(\RR); }
  }

  \draw[match] (L1)--(R1);
  \draw[match] (L2)--(R2);
  \draw[match] (L3)--(R3);
  \draw[match] (L4)--(R6);

\end{tikzpicture}
\caption{Case A (admissible): $\nu=4$. }
\label{fig:match_caseA_side}
\end{subfigure}\hfill
\begin{subfigure}[t]{0.485\columnwidth}
\centering
\begin{tikzpicture}[x=1cm,y=1cm,baseline]
  \tikzset{
    dot/.style={circle,fill,inner sep=1.2pt},
    Llab/.style={anchor=east,font=\scriptsize},
    Rlab/.style={anchor=west,font=\scriptsize},
    edge/.style={line width=0.45pt,draw=black!35},
    match/.style={line width=1.05pt,draw=blue}
  }

  \def\xL{0.0}
  \def\xR{1.85}
  \def\dxLab{0.18}

  \node[dot] (L1) at (\xL,3.0) {}; \node[Llab] at ($(L1)+(-\dxLab,0)$) {$(12,1)$};
  \node[dot] (L2) at (\xL,2.0) {}; \node[Llab] at ($(L2)+(-\dxLab,0)$) {$(12,2)$};
  \node[dot] (L3) at (\xL,1.0) {}; \node[Llab] at ($(L3)+(-\dxLab,0)$) {$(23,1)$};
  \node[dot] (L4) at (\xL,0.0) {}; \node[Llab] at ($(L4)+(-\dxLab,0)$) {$(23,2)$};

  \node[dot] (Rr1) at (\xR,3.2) {}; \node[Rlab] at ($(Rr1)+(\dxLab,0)$) {$u_{r,1}$};
  \node[dot] (Rt1) at (\xR,2.6) {}; \node[Rlab] at ($(Rt1)+(\dxLab,0)$) {$u_{\theta,1}$};

  \node[dot] (Rr2) at (\xR,1.8) {}; \node[Rlab] at ($(Rr2)+(\dxLab,0)$) {$u_{r,2}$};
  \node[dot] (Rt2) at (\xR,1.2) {}; \node[Rlab] at ($(Rt2)+(\dxLab,0)$) {$u_{\theta,2}$};

  \node[dot] (Rr3) at (\xR,0.4) {}; \node[Rlab] at ($(Rr3)+(\dxLab,0)$) {$u_{r,3}$};
  \node[dot] (Rt3) at (\xR,-0.2) {}; \node[Rlab] at ($(Rt3)+(\dxLab,0)$) {$u_{\theta,3}$};

  \foreach \LL in {L1,L2}{
    \foreach \RR in {Rr1,Rr2}{ \draw[edge] (\LL)--(\RR); }
  }
  \foreach \LL in {L3,L4}{
    \foreach \RR in {Rr2,Rr3}{ \draw[edge] (\LL)--(\RR); }
  }

  \draw[match] (L1)--(Rr1);
  \draw[match] (L3)--(Rr2);
  \draw[match] (L4)--(Rr3);

\end{tikzpicture}
\caption{Case B (non-admissible): $\nu=3<4$. }
\label{fig:match_caseB_side}
\end{subfigure}

\caption{Bipartite graphs associated with the sparsity pattern of $A_\tau(x)$. Gray edges show structural nonzeros; colored edges show a maximum matching.}
\label{fig:matching_side_by_side}
\end{figure}

\begin{figure}[!h]
\centering
\begin{subfigure}[t]{0.485\columnwidth}
\centering
\includegraphics[width=\linewidth]{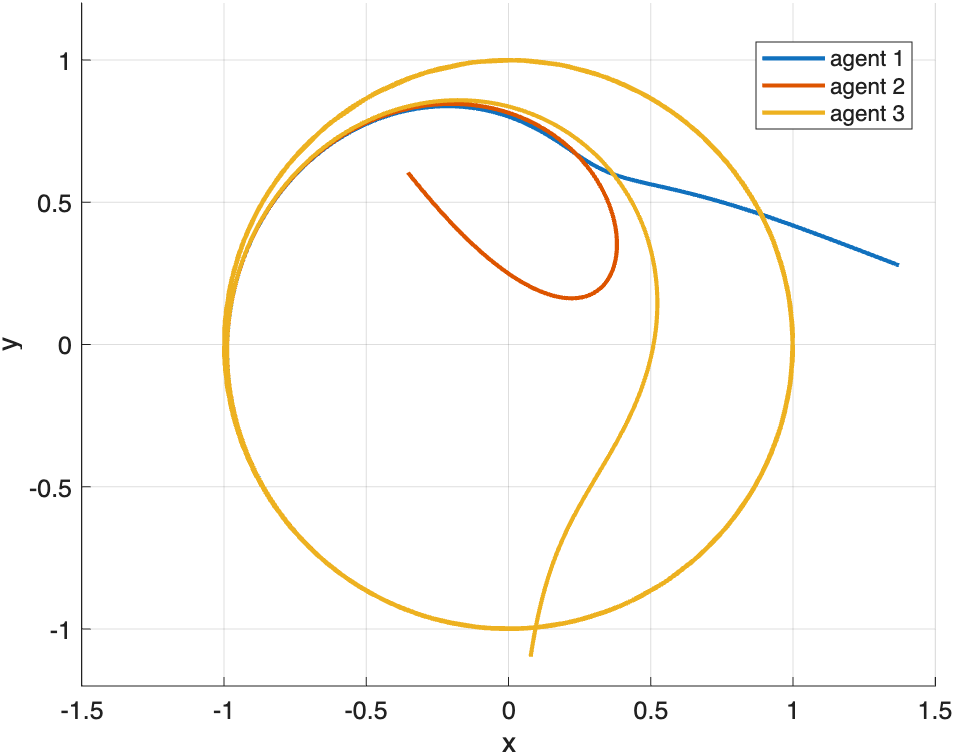}
\caption{Case A (admissible): phase-plane trajectories synchronize.}
\label{fig:sim_phase_A}
\end{subfigure}\hfill
\begin{subfigure}[t]{0.485\columnwidth}
\centering
\includegraphics[width=\linewidth]{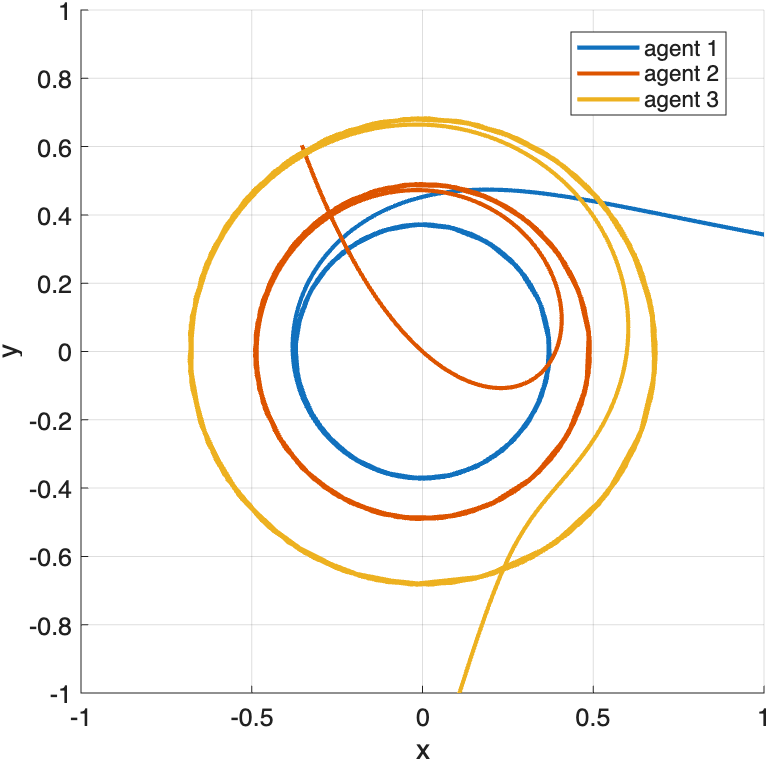}
\caption{Case B (non-admissible): phase-plane trajectories remain phase-shifted.}
\label{fig:sim_phase_B}
\end{subfigure}
\caption{Phase-plane trajectories for the oscillator network under the two actuation patterns.}
\label{fig:sim_phase_AB}
\end{figure}


\section{Concluding Remarks}
This paper developed an edge-driven geometric framework for diffusive coupling of heterogeneous control-affine networks. The design philosophy is to construct a stabilizing control in the edge space and then generate a feasible agent input through an appropriate lift from edge space to node space.  We identify \emph{admissibility} as a structural condition linking the interconnection graph with the actuation limits of the agents.
We further provided generic/combinatorial certificates, based on structured rank and maximum matchings, that predict when edge-driven diffusive designs can be realized and when \textcolor{black}{synchronization} is obstructed.

A  next step is to extend this geometric admissibility   viewpoint to settings where \emph{both} layers are nonlinear: i.e.,  with nonlinear measurement maps (e.g., distances, bearings, or other nonlinear relative outputs) and nonlinear agent dynamics.

\bibliographystyle{IEEEtran}
\bibliography{formation}

\end{document}